\documentclass{article}

\usepackage{amssymb}
\usepackage{amsmath}
\usepackage{amsthm}
\usepackage[dvips]{graphicx}

\theoremstyle{theorem} \newtheorem{defin}{Definition}
\theoremstyle{theorem} \newtheorem{theo}[defin]{Theorem}
\theoremstyle{theorem} \newtheorem{prop}[defin]{Proposition}
\theoremstyle{theorem} \newtheorem{coro}[defin]{Corollary}
\theoremstyle{theorem} \newtheorem{rema}[defin]{Remark}
\theoremstyle{theorem} \newtheorem{lemma}[defin]{Lemma}

\begin{document}

\title{A non-periodic and two-dimensional example of elliptic homogenization}

\author{Jens Persson}

\maketitle

\begin{abstract}
	The focus in this paper is on elliptic homogenization of a certain kind of possibly non-periodic problems.
	A non-periodic and two-dimensional example is studied, where we numerically illustrate the homogenized matrix. 
\end{abstract}

\section{Introduction}

\paragraph{Background.} When studying the microscale behavior (beyond the reach of numerical solution methods) of
physical systems, one is naturally lead to the concept of homogenization, i.e., the theory of the convergence of
sequences of partial differential equations.

The homogenization of periodic structures using the two-scale convergence technique is well-established due to the
pioneering work by Gabriel Nguetseng~\cite{Ngu89} and the further development work by Gr\'{e}goire
Allaire~\cite{All92}. Generalizations of the two-scale convergence technique have been developed independently by,
e.g., Maria Lu\'{i}sa Mascarenhas and Anca-Maria Toader~\cite{MasToa01} (scale convergence), Gabriel
Nguetseng~\cite{Ngu03,Ngu04} ($\Sigma$-convergence), and Anders Holmbom, Jeanette Silfver, Nils Svanstedt and Niklas
Wellander~\cite{HolPhD96,HSSW06} (``generalized'' two-scale convergence).

A simple but possibly powerful method of analyzing non-periodic structures is the $\lambda$-scale convergence technique
introduced by Anders Holmbom and Jeanette Silfver~\cite{HolSil06}. $\lambda$-scale convergence is scale convergence in
the special case of using the Lebesgue (i.e., $\lambda$) measure and test functions periodic in the second argument
\cite{HolSil06}. Homogenization techniques based on this approach are developed in the doctoral thesis~\cite{SilPhD07}
of Jeanette Silfver. These results are the point of departure for the main contributions in this paper.

\paragraph{Organization of the paper.} In Section~\ref{sec:convfunseq} we look at the convergence for sequences of
functions. We start by stating the definition of the traditional notion of two-scale convergence as introduced by
Gabriel Nguetseng~\cite{Ngu89}. Then we move on to generalizations of this convergence mode, namely ``generalized''
two-scale convergence~\cite{HolPhD96,HSSW06} and scale convergence~\cite{MasToa01}. We conclude the section by
introducing $\lambda$-scale convergence and the important notion of asymptotically uniformly distributed
sequences~\cite{HolSil06,SilPhD07}.

Section~\ref{seq:convseqpartder} deals with the convergence for sequences of partial derivatives. We first look at how
the two-scale convergence works for partial derivatives in the periodic case, and then we consider the more general
case of $\lambda$-scale convergence of sequences of partial derivatives.

The concept of homogenization is introduced in Section~\ref{sec:ellhomo}, which begins with stating the definition of
H-convergence~\cite{Mur78}, i.e., the generalization of Sergio Spagnolo's G-convergence of sequences of symmetric
matrices~\cite{Spa67,Spa68}. We give a theorem on the homogenization of a sequence of periodic matrices. We then
introduce the important ``type-$\mathrm{H}^\zeta_X$'' property, which is employed in the end of the section when
formulating a theorem on the homogenization of $\lambda$-structures~\cite{SilPhD07}.  

In Section~\ref{sec:nonpertwodimex} we specifically study a non-periodic and two-dimensional example of a
$\lambda$-structure with the property of not only having a periodic direction, but also having oscillations with a
monotonically increasing frequency in one, non-periodic direction. We formulate and prove a series of proposition which
are needed in order to prove the main result of this paper, Theorem~\ref{theo:lambdahomoex}, in which we present a
homogenization result for the considered $\lambda$-structure in the form of a homogenized matrix and a governing local
problem.
			
The concluding Section~\ref{sec:numill} illustrates the results achieved in Section~\ref{sec:nonpertwodimex}. We
numerically solve the local problem to obtain the non-constant homogenized matrix, and we heuristically explain why it
is isotropic on a line along the periodic direction.

\paragraph{Notations.} The following more or less handy notations are employed:

The $N$-tuple $(\xi_1,\ldots,\xi_N)$ in $\mathbb{R}^N$ is denoted $(\xi_i)_{i=1,\ldots,N}$, or $\xi$ whenever
convenient. Similarly, $(m_{ij})_{i,j=1,\ldots,N}$, or simply the mere majuscule $M$ when handy, denotes an
$N \times N$ matrix. A bold dot, i.e. $\, \boldsymbol{\cdot} \,$, represents a non-fixed variable, e.g.,
$\phi(\, \boldsymbol{\cdot} \, ,y)$ is the same as the function $x \mapsto \phi(x,y)$ where $y$ is held fixed as a
parameter. We will also allow expressions like, e.g., $k \, \boldsymbol{\cdot} \,$ meaning $x \mapsto kx$. In this
paper, $\Omega$ is always an open bounded non-empty subset of $\mathbb{R}^N$ and, if nothing else is stated, $Y$ is the unit cube
$(0,1)^N$ in $\mathbb{R}^N$. The $N$-tuple $\bigl( \tfrac{\partial}{\partial x_i} \bigr)_{i=1,\ldots,N}$ of partial
derivative operators is denoted by the symbol $\nabla$. Following, e.g., \cite{CioDon99}, the function space
$W_{\mathrm{per}} (Y)$ denotes the subspace of functions in $H^1_{\mathrm{per}} (Y)$ with vanishing mean value.

\section{Convergence for sequences of functions}
\label{sec:convfunseq}

The two-scale convergence method was introduced in 1989 by Gabriel Nguetseng~\cite{Ngu89}, and a modern formulation is
given by Definition~\ref{def:twoscaleconv}~\cite{LNgW02}.
\begin{defin}
	\label{def:twoscaleconv}
	A sequence $\{ u^h \}$ in $L^2(\Omega)$ is said to two-scale converge to the limit $u_0 \in L^2 (\Omega \times Y)$
	if, for any $v \in L^2 \bigl( \Omega; \, C_{\mathrm{per}} (Y) \bigr)$,
	\begin{equation*}
		\lim_{h \rightarrow \infty} \int\limits_{\Omega} u^h (x) v (x,hx) \, \mathrm{d} x
		= \int\limits_{\Omega} \int\limits_Y u_0 (x,y) v (x,y) \, \mathrm{d} y \, \mathrm{d} x.  
	\end{equation*}
\end{defin}
An important property of the two-scale convergence is given by Proposition~\ref{prop:twoscaleweak}~\cite{LNgW02}.
\begin{prop}
	\label{prop:twoscaleweak}
	If $\{ u^h \}$ two-scale converges to $u_0$, and $v \in L^2 \bigl( \Omega; \, C_{\mathrm{per}} (Y) \bigr)$, then
	\begin{align*}
		u^h								&\rightharpoonup \int\limits_Y u_0(\, \boldsymbol{\cdot} \, ,y) \, \mathrm{d} y
											\quad \textrm{in } L^2 (\Omega)
		\intertext{and}
		v(\, \boldsymbol{\cdot} \,
		,h \, \boldsymbol{\cdot} \,)	&\rightharpoonup \int\limits_Y v(\, \boldsymbol{\cdot} \, ,y) \, \mathrm{d} y
											\quad \textrm{in } L^2 (\Omega).
	\end{align*}
\end{prop}
Analogous to a corresponding compactness result for weak convergence, we have
Theorem~\ref{th:twoscalecomp}~\cite{LNgW02}.
\begin{theo}
	\label{th:twoscalecomp}
	Every bounded sequence in $L^2(\Omega)$ has a subsequence which two-scale converges.
\end{theo}
Introducing sequences of operators $\tau^h$ defined below we may extend Definition~\ref{def:twoscaleconv} of two-scale
convergence to a generalized version according to Definition~\ref{def:gentwoscale}~\cite{HolPhD96,HSSW06}.
\begin{defin}
	\label{def:gentwoscale}
	Assume that $Y$ is an open bounded subset of $\mathbb{R}^M$. Let $X \subset L^2 (\Omega \times Y)$ be a linear
	space and
	\begin{equation*}
		\tau^h \, : \, X \rightarrow L^2 (\Omega)
	\end{equation*}
	linear operators. A sequence $\{ u^h \}$ in $L^2 (\Omega)$ is said to two-scale converge to
	$u_0 \in L^2 (\Omega \times Y)$ with respect to $\{ \tau^h \}$ if, for any $v \in X$,
	\begin{equation*}
		\lim_{h \rightarrow \infty} \int\limits_{\Omega} u^h (x) ( \tau^h v ) (x) \, \mathrm{d} x
		= \int\limits_{\Omega} \int\limits_Y u_0 (x,y) v (x,y) \, \mathrm{d} y \, \mathrm{d} x.
	\end{equation*}
\end{defin}
In order to achieve a compactness result like Theorem~\ref{th:twoscalecomp}, we need Definition~\ref{def:compatible}.
\begin{defin}
	\label{def:compatible}
	Assume that $Y$ is an open bounded subset of $\mathbb{R}^M$. Let $X \subset L^2 (\Omega \times Y)$ be a normed
	space and
	\begin{equation*}
		\tau^h \, : \, X \rightarrow L^2 (\Omega)
	\end{equation*}
	linear operators. Then $\{ \tau^h \}$ is two-scale compatible with respect to $X$ if there exists $C > 0$
	independent of $h$ such that, for any $v \in X$,
	\begin{align*}
		\lim_{h \rightarrow \infty} \lVert \tau^h v \rVert_{L^2(\Omega)}
				&\leqslant C \lVert v \rVert_{L^2(\Omega \times Y)}	
		\intertext{and}
		\lVert \tau^h v \rVert_{L^2(\Omega)}
				&\leqslant C \lVert v \rVert_X.
	\end{align*}
	Furthermore, $X$ is called admissible with respect to $\{ \tau^h \}$.
\end{defin}
Using Definition~\ref{def:compatible}, we have the compactness result according to
Theorem~\ref{th:gentwoscalecomp}~\cite{HolPhD96,HSSW06}.
\begin{theo}
	\label{th:gentwoscalecomp}
	Assume that $Y$ is an open bounded subset of $\mathbb{R}^M$. Let $\{ \tau^h \}$ be two-scale compatible with
	respect to $X$, a separable Banach space dense in $L^2(\Omega \times Y)$. Then every bounded sequence $\{ u^h \}$
	in $L^2(\Omega)$ has a subsequence that two-scale converges with respect to $\{ \tau^h \}$.
\end{theo}
Using weak convergences, we can introduce a stronger version of two-scale compatibility than presented in
Definition~\ref{def:compatible}, see Definition~\ref{def:strongcompatible}.
\begin{defin}
	\label{def:strongcompatible}
	Assume that $Y$ is an open bounded subset of $\mathbb{R}^M$. Let $\{ \tau^h \}$ be two-scale compatible with
	respect to the linear space $X \subset L^2(\Omega \times Y)$, and let $\{ u^h \}$ be a bounded sequence in
	$L^2 (\Omega)$ two-scale converging to $u_0$. Then $\{ \tau^h \}$ is strongly two-scale compatible if, for any
	$v \in X$,
	\begin{align*}
		u^h			&\rightharpoonup \int\limits_Y u_0(\, \boldsymbol{\cdot} \, ,y) \, \mathrm{d} y
						\quad \textrm{in } L^2 (\Omega)
		\intertext{and}
		\tau^h v	&\rightharpoonup \int\limits_Y v(\, \boldsymbol{\cdot} \, ,y) \, \mathrm{d} y
						\quad \textrm{in } L^2 (\Omega).
	\end{align*}
\end{defin}
Compare Definition~\ref{def:strongcompatible} with the result of Proposition~\ref{prop:twoscaleweak}.

An alternative generalization of two-scale convergence is scale convergence introduced by Maria Lu\'{i}sa Mascarenhas
and Anca-Maria Toader in 2001~\cite{MasToa01}. The definition of scale convergence is according to
Definition~\ref{def:scaleconv}.
\begin{defin}
	\label{def:scaleconv}
	Assume that $Y$ is a metrizable compact space, $\mu$ a Young measure on $\Omega \times Y$, and
	$L^2_{\mu} (\Omega \times Y)$ is the space of all functions with $\mu$-integrable square. Furthermore, let
	$\{ \alpha^h \}$ be a sequence of $\mu$-measurable functions $\alpha^h \, : \, \Omega \rightarrow Y$. A sequence
	$\{ u^h \}$ in $L^2 (\Omega)$ is said to scale converge to $u_0 \in L^2_{\mu} (\Omega \times Y)$ with respect to
	$\{ \alpha^h \}$ if, for any $v \in L^2 \bigl( \Omega; \, C(Y) \bigr)$,
	\begin{equation*}
		\lim_{h \rightarrow \infty} \int\limits_{\Omega} u^h (x) v \bigl( x, \alpha^h (x) \bigr) \, \mathrm{d} x
		= \int\limits_{\Omega \times Y} u_0 (x,y) v (x,y) \, \mathrm{d} \mu (x,y).
	\end{equation*}
\end{defin}
A special case of scale convergence is achieved by choosing the same measure and same class of admissible test
functions as in two-scale convergence, namely the Lebesgue measure $\lambda$ and
$L^2 \bigl( \Omega; \, C_{\mathrm{per}} (Y) \bigr)$, respectively. This leads to
Definition~\ref{def:lambdascale}~\cite[Definition~27]{SilPhD07} (see also~\cite[Definition~11]{HolSil06}).  
\begin{defin}
	\label{def:lambdascale}
	Assume that $Y$ is the unit cube in $\mathbb{R}^M$. Furthermore, let $\{ \alpha^h \}$ be a sequence of functions
	$\alpha^h \, : \, \Omega \rightarrow Y$. A sequence $\{ u^h \}$ in $L^2(\Omega)$ is said to $\lambda$-scale
	converge to $u_0 \in L^2(\Omega \times Y )$ with respect to $\{ \alpha^h \}$ if, for any
	$v \in L^2 \bigl( \Omega; \, C_{\mathrm{per}} (Y) \bigr)$,
	\begin{equation*}
		\lim_{h \rightarrow \infty} \int\limits_{\Omega} u^h (x) v\bigl( x,\alpha^h(x) \bigr) \, \mathrm{d} x
		= \int\limits_{\Omega} \int\limits_Y u_0 (x,y) v (x,y) \, \mathrm{d} y \, \mathrm{d} x.  
	\end{equation*}
\end{defin}
A thorough treatment of how $\{ \alpha^h \}$ could be chosen to obtain strong two-scale compatibility is found in the
doctoral thesis~\cite[Subsection~2.4.2]{SilPhD07} of Jeanette Silfver (see also \cite{HolSil06}). We reproduce the
results below.

Following~\cite[Subsection~2.4.2]{SilPhD07}, we omit cases where $M \neq N$ letting
\begin{equation*}
	\alpha^h \, : \, \mathbb{R}^N \rightarrow \mathbb{R}^N
\end{equation*}
be a continuous bijection. Furthermore, $\{ \bar{Y}^j \}_{j=1}^{\infty}$ is a covering of $\mathbb{R}^N$ with unit
cubes, and $\{ \bar{Y}^j_k \}_{k=1}^{n^N}$ a covering of $\bar{Y}^j$ with cubes with side lengths $\tfrac{1}{n}$ in
such a way that $Y^j_k$ are $Y$-periodic repetitions of a cube $Y_k \subset Y$. Assuming that
$\Omega \subset \mathbb{R}^N$ has a Lipschitz boundary, we define
\begin{equation*}
	\Omega^h_j = (\alpha^h)^{-1} (\bar{Y}^j) \cap \Omega .
\end{equation*}
We also assume, for each $h$, that there exists a finite set $q(h) \subset \mathbb{Z}_+$ such that
$\Omega = \cup_{j \in q(h)}\Omega^h_j$. Similarly, introduce
\begin{equation*}
	\Omega^h_{j,k} = (\alpha^h)^{-1} (\bar{Y}^j_k) \cap \Omega .
\end{equation*}
Finally, it is assumed that $\Omega^h_j \subset N_{r(h)}(x_{h,j})$ for some ball $N_{r(h)}(x_{h,j})$, centered at some
$x_{h,j} \in \Omega^h_j$, with radius $r(h) \rightarrow 0$ as $h \rightarrow \infty$. Given the setting above,
Definition~\ref{def:AUD} makes sense~\cite[Definition~28]{SilPhD07} (see also \cite[Definition~12]{HolSil06}).
\begin{defin}
	\label{def:AUD}
	Suppose that for all cubes $Y_k \subset Y$ and any $\Omega^h_j$, $j \in q(h)$, such that
	$(\alpha^h)^{-1} (Y^j) \cap \Omega$ does not intersect $\partial \Omega$,
	\begin{equation}
		\label{eq:AUDequation}
		\left| \frac{\lambda(\Omega^h_{j,k})}{\lambda(\Omega^h_j)} - \lambda(Y_k) \right| < \epsilon
	\end{equation}
	where $\epsilon = \epsilon (h) \rightarrow 0$ for $h \rightarrow \infty$. Then $\{ \alpha^h \}$ is said to be
	asymptotically uniformly distributed on $\Omega$.
\end{defin}
This is sufficient to obtain the strong two-scale compatibility as promised, see
Proposition~\ref{prop:lambdastronglytwoscale}~\cite[Proposition~30]{SilPhD07} (see
also~\cite[Proposition~15]{HolSil06}).
\begin{prop}
	\label{prop:lambdastronglytwoscale}
	Assume that $\Omega$ has a Lipschitz continuous boundary. Let $\{ \alpha^h \}$ be asymptotically uniformly
	distributed on $\Omega$. Then $\{ \tau^h \}$ defined by
	\begin{equation*}
		\tau^h v = v \bigl( \, \boldsymbol{\cdot} \, , \alpha^h (\, \boldsymbol{\cdot} \,) \bigr)
	\end{equation*}
	is strongly two-scale compatible with respect to $L^2 \bigl( \Omega; \, C_{\mathrm{per}} (Y) \bigr)$.
\end{prop}
Note here that the admissible space in $L^2 (\Omega \times Y)$ is in this case
$L^2 \bigl( \Omega; \, C_{\mathrm{per}} (Y) \bigr)$. As a consequence of Proposition~\ref{prop:lambdastronglytwoscale}, we
have Corollary~\ref{cor:sublambdascale}~\cite[Corollary~31]{SilPhD07} (see also~\cite[Corollary~16]{HolSil06}).
\begin{coro}
	\label{cor:sublambdascale}
	Assume that $\Omega$ has a Lipschitz continuous boundary. Let $\{ \alpha^h \}$ be asymptotically uniformly
	distributed on $\Omega$, and let $\{ u^h \}$ strongly converge to $u$ in $L^2 (\Omega)$. Then, up to a subsequence,
	$\{ u^h \}$ $\lambda$-scale converges to $u$ with respect to $\{ \alpha^h \}$.
\end{coro}
Note that the second scale dependence of the $\lambda$-scale limit vanishes, just like how it works for two-scale
limits in the case of strong convergence~\cite{LNgW02}.

\section{Convergence for sequences of $N$-tuples of partial derivatives}
\label{seq:convseqpartder}

Since sequences of partial differential equations in the context of homogenization typically involve $N$-tuples of
partial derivatives of solutions, i.e., $\nabla u^h$, we need to investigate the two-scale limits for these. Indeed,
for traditional periodic two-scale convergence, we have Proposition~\ref{prop:twoscalegrad}~\cite{All92,Ngu89}.
\begin{prop}
	\label{prop:twoscalegrad}
	Let $\{ u^h \}$ be a bounded sequence in $H^1 (\Omega)$ such that the strong limit in $L^2 (\Omega)$ is $u$. Then,
	up to a subsequence, there exists $u_1 \in L^2 \bigl( \Omega; \, W_{\mathrm{per}} (Y) \bigr)$ such that
	for any $v \in L^2 \bigl( \Omega; \, C_{\mathrm{per}} (Y)^N \bigr)$,
	\begin{multline*}
		\lim_{h \rightarrow \infty} \int\limits_{\Omega} \sum_{i=1}^N \,
			\frac{\partial u^h}{\partial x_i}(x) \, v_i(x,hx) \, \mathrm{d} x \\
		= \int\limits_{\Omega} \int\limits_Y \sum_{i=1}^N \, \biggl( \frac{\partial u}{\partial x_i}(x)
			+ \frac{\partial u_1}{\partial y_i}(x,y) \biggr) v_i(x,y) \, \mathrm{d} y \, \mathrm{d} x .
	\end{multline*}
\end{prop}
It should be noted here that bounded sequences in $H^1 (\Omega)$ strongly converge in $L^2 (\Omega)$, and we know that
strongly convergent sequences in $L^2 (\Omega)$ have a subsequence with a two-scale limit with vanishing second scale.

A deciding step towards the homogenization of certain non-periodic problems is to prove the corresponding result for
$\lambda$-scale convergence. We have Proposition~\ref{prop:lambdatwoscalegrad} proved by Jeanette Silfver
in~\cite[Proposition~35]{SilPhD07} (see also~\cite[Section~4]{HolSil06}).
\begin{prop}
	\label{prop:lambdatwoscalegrad}
	Let  $X$ be a Banach space for which
	\begin{equation*}
		\mathcal{D} \bigl( \Omega; \, C_{\mathrm{per}}^{\infty} (Y) \bigr) \subset X
			\subset L^2 \bigl( \Omega; \, L^2_{\mathrm{per}} (Y) \bigr),
	\end{equation*}
	and $\{ \alpha^h \}$ a sequence of functions
	\begin{equation*}
		\alpha^h \, : \, \mathbb{R}^N \rightarrow \mathbb{R}^N
	\end{equation*}
	which are continuous and bijective. We assume also that $\{ \tau^h \}$ defined by
	\begin{equation*}
		\tau^h v = v \bigl( \, \boldsymbol{\cdot} \, ,\alpha^h(\, \boldsymbol{\cdot} \,) \bigr)
	\end{equation*}
	is strongly two-scale compatible with respect to $X$. Furthermore, let
	$Z \subset L^2 \bigl( \Omega; \, L^2_{\mathrm{per}} (Y)^N \bigr)$ be a Banach space such that
	$Z \cap \mathcal{D} \bigl( \Omega; \, C_{\mathrm{per}}^{\infty} (Y)^N \bigr)$ is dense in $Z$, and let $Z^{\perp}$ be
	its orthogonal complement in $L^2 \bigl( \Omega; \, L^2_{\mathrm{per}} (Y)^N \bigr)$. Assume also that, for any
	$v \in Z \cap \mathcal{D} \bigl( \Omega; \, C_{\mathrm{per}}^{\infty} (Y)^N \bigr)$,
	\begin{equation}
		\label{eq:traceofprod}
		\sum_{i=1}^N \sum_{j=1}^N \, \frac{\partial v_i}{\partial y_j}\bigl( \, \boldsymbol{\cdot} \,
		,\alpha^h(\, \boldsymbol{\cdot} \,) \bigr) \, \frac{\partial \alpha^h_j}{\partial x_i}
		\rightharpoonup 0 \quad \textrm{in } L^2 (\Omega).
	\end{equation}
	
	Then, for any bounded sequence $\{ u^h \}$ in $H^1 (\Omega)$ there exists a subsequence such that
	$u^h \rightarrow u$ in $L^2 (\Omega)$ and, for any $v \in X$,
	\begin{equation}
		\label{eq:vansecsca}
		\lim_{h \rightarrow \infty} \int\limits_{\Omega} u^h (x) v\bigl( x,\alpha^h(x) \bigr) \, \mathrm{d} x
		= \int\limits_{\Omega} \int\limits_Y u(x) v(x,y) \, \mathrm{d} y \, \mathrm{d} x, 
	\end{equation}
	and there exists $w_1 \in Z^{\perp}$ such that, for any $v \in X^N$,
	\begin{multline}
		\label{eq:twoscagrad}
		\lim_{h \rightarrow \infty} \int\limits_{\Omega} \sum_{i=1}^N \, \frac{\partial u^h}{\partial x_i}(x)
			\, v_i\bigl( x,\alpha^h(x) \bigr) \, \mathrm{d} x \\
		= \int\limits_{\Omega} \int\limits_Y \sum_{i=1}^N \, \biggl( \frac{\partial u}{\partial x_i}(x)
			+ w_{1,i} (x,y) \biggr) v_i(x,y) \, \mathrm{d} y \, \mathrm{d} x.
	\end{multline}
\end{prop}
\begin{proof}
	Property \eqref{eq:vansecsca} is given by Corollary~\ref{cor:sublambdascale}. By using Green's formula twice, using
	assumption \eqref{eq:traceofprod}, and utilizing density, property \eqref{eq:twoscagrad} follows from
	orthogonality.
\end{proof}
\begin{rema}
	The original proof of Jeanette Silfver with all details is found in~\cite[Subsection~2.4.2]{SilPhD07}.
\end{rema}

\section{Elliptic homogenization}
\label{sec:ellhomo}

In this section we present a homogenization result for elliptic problems governed by the sequence $\{ \alpha^h \}$
introduced earlier in this paper. Such results where first published in~\cite{SilPhD07}, and in
Sections~\ref{sec:nonpertwodimex} and~\ref{sec:numill} we present a special case in two dimensions and perform a
numerical experiment, respectively.

The foundation upon which modern homogenization rests was erected in 1967-68 when
Sergio Spagnolo developed the concept of G-convergence of sequences of symmetric matrices~\cite{Spa67,Spa68}, which has
later been generalized by Fran\c{c}ois Murat to H-convergence where the symmetry assumption is dropped, but at the cost
of an imposed requirement on the sequence of flows~\cite{Mur78}. We begin by introducing a space of matrix valued
functions according to Definition~\ref{def:Mspace}.
\begin{defin}
	\label{def:Mspace}
	If $\infty > r \geqslant s > 0$ and $\mathcal{O} \subset \mathbb{R}^N$ is open, the matrix valued function
	$M \in L^{\infty} (\mathcal{O})^{N \times N}$ is said to belong to $\mathcal{M} (r,s; \, \mathcal{O})$ if the
	system of structural conditions
	\begin{equation*}
		\left\{
			\begin{aligned}
				\sup_{|\xi|=1} \biggl| \Bigl( \sum_{j=1}^N \, m_{ij} \xi_j \Bigr)_{i=1,\ldots,N} \biggr|
					&\leqslant r, \quad \textrm{(bounded)} \\
				\inf_{|\xi|=1} \sum_{i=1}^N \sum_{j=1}^N \, \xi_i m_{ij} \xi_j
					&\geqslant s, \quad \textrm{(coercive)}
			\end{aligned}
		\right.
	\end{equation*}
	is satisfied a.e.~in $\mathcal{O}$. Furthermore, the space $\mathcal{M}_{\mathrm{per}}(r,s; \, Y)$ consists of
	those functions in $\mathcal{M} \bigl( r,s; \, \mathbb{R}^N \bigr)$ which are $Y$-periodic.
\end{defin}
We can now give Definition~\ref{def:ellHconv} defining the notion of H-convergence.
\begin{defin}
	\label{def:ellHconv}
	Let $\{ A^h \}$ be a sequence in $\mathcal{M}(r,s; \, \Omega)$, and let $B \in \mathcal{M} (r',s'; \, \Omega )$.
	Furthermore, assume that, for any $f \in H^{-1} (\Omega)$, the sequence of solutions $\{ u^h \}$ in
	$H^1_0 (\Omega)$ to the sequence of problems
	\begin{equation}
		\label{eq:PDEseq}
		\left\{
			\begin{aligned}
				- \sum_{i=1}^N \sum_{j=1}^N \, \frac{\partial}{\partial x_i} \biggl( a^h_{ij} \,
					\frac{\partial u^h}{\partial x_j} \biggr)	&= f	\qquad \textrm{in } \Omega \\
				u^h 											&= 0	\qquad \textrm{on } \partial \Omega
			\end{aligned}
		\right.
	\end{equation}
	satisfies
	\begin{equation}
		\label{eq:limits}
		\left\{
			\begin{aligned}
				u^h						&\rightharpoonup u			\qquad \textrm{in } H^1_0 (\Omega) \\
				\biggl( \sum_{j=1}^N \, a^h_{ij} \, \frac{\partial u^h}{\partial x_j} \biggr)_{i=1,\ldots,N}
										&\rightharpoonup \biggl( \sum_{j=1}^N \, b_{ij} \,
					\frac{\partial u}{\partial x_j} \biggr)_{i=1,\ldots,N} \quad \textrm{in } L^2 (\Omega)^N
			\end{aligned}
		\right. ,
	\end{equation}
	where $u$ is the unique solution to
	\begin{equation}
		\label{eq:homogenized}
		\left\{
			\begin{aligned}
				- \sum_{i=1}^N \sum_{j=1}^N \, \frac{\partial}{\partial x_i} \biggl( b_{ij} \,
					\frac{\partial u}{\partial x_j} \biggr)	&= f	\qquad \textrm{in } \Omega \\
				u 											&= 0	\qquad \textrm{on } \partial \Omega
			\end{aligned}
		\right. .
	\end{equation}
	Then $\{ A^h \}$ H-converges to $B$.
\end{defin}
\begin{rema}
	In the literature, H-convergence is often called G-convergence. In this paper, we clearly separate the general
	notion of H-convergence from the special case of G-convergence which deals with symmetric matrices exclusively.
\end{rema}
When an H-limit has been found, one has homogenized the sequence of partial differential equations \eqref{eq:PDEseq},
and \eqref{eq:homogenized} is the homogenized problem. For sequences of matrices which are periodic, we get
Theorem~\ref{theo:elliptichomo} for the homogenization in this case~\cite{Ngu89,All92}.
\begin{theo}
	\label{theo:elliptichomo}
	Suppose $A \in \mathcal{M}_{\mathrm{per}}(r,s; \, Y)$. Then $\{ A(h\, \boldsymbol{\cdot} \,) \}$
	H-converges to $B$ given by
	\begin{equation*}
		\bigl( b_{ij} \bigr)_{i,j=1,\ldots,N} = \Biggl( \int\limits_Y \sum_{k=1}^N \, a_{ik}(y) \biggl( \delta_{kj}
			+ \frac{\partial z_j}{\partial y_k}(y) \biggr) \, \mathrm{d} y \Biggr)_{i,j=1,\ldots,N} ,
	\end{equation*}
	where $z \in W_{\mathrm{per}} (Y)^N$ uniquely solves the local problem
	\begin{equation}
	\label{eq:localproblem}
		- \Biggl( \sum_{i=1}^N \sum_{k=1}^N \, \frac{\partial}{\partial y_i} \! \biggl\lgroup \! a_{ik}
			\Bigl( \delta_{kj} + \frac{\partial z_j}{\partial y_k} \Bigr) \! \biggr\rgroup \! \Biggr)_{j=1,\ldots,N}
				= 0 \qquad \textrm{in } Y.
	\end{equation}
\end{theo}
We note that it is the term containing the local problem solution $z$ which makes the H-limit $B$ to deviate from the
average of $A$ over each periodicity cell $Y$. It is due to the $u_1$ term in the two-scale convergence for partial
derivatives and is obtained through a simple separation of variables of $u_1$ with $\nabla u$ as the $x$-dependent part
and $z$ as the $y$-dependent part.

In the remainder of this paper, $A^h$ in the sequence  of problems \eqref{eq:PDEseq} will be the composition $A \circ \alpha^h$.

Below we present the result on non-periodic homogenization from~\cite{SilPhD07} announced in the beginning of this
section. We start by giving Definition~\ref{def:typeH}~\cite[Definition~51]{SilPhD07}.
\begin{defin}
	\label{def:typeH}
	The sequence $\{ \alpha^h \}$ is of type $\mathrm{H}^\zeta_X$ if the following three conditions hold:
	\begin{itemize}
		\item[(i)]		The sequence $\{ \tau^h \}$ of operators $\tau^h \, : \, X \rightarrow L^2 (\Omega)$ defined by
						\begin{equation*}
							\tau^h v = ´v \bigl( \, \boldsymbol{\cdot} \,
								, \alpha^h(\, \boldsymbol{\cdot} \,) \bigr), \qquad v \in X
						\end{equation*}
						is strongly two-scale compatible with respect to a Banach space $X$ for which
						\begin{equation*}
							\mathcal{D} \bigl( \Omega; \, C^{\infty}_{\mathrm{per}} (Y) \bigr) \subset X
								\subset L^2 \bigl( \Omega; \, L^2_{\mathrm{per}} (Y) \bigr).
						\end{equation*}
		\item[(ii)]		There exists a sequence $\{ p^h \}$ strongly convergent to zero in $H^1 (\Omega)$ such that
						\begin{equation*}
							p^h \, \Bigl( \frac{\partial \alpha^h_j}{\partial x_i} \Bigr)_{i,j=1,\ldots,N}
								\rightarrow \bigl( \pi_{ij} \bigr)_{i,j=1,\ldots,N}
									\qquad \textrm{in } L^2(\Omega)^{N \times N},
						\end{equation*}
						where $\Pi$ is a diagonal matrix with
						\begin{equation*}
							\bigl( \pi_{ii} \bigr)_{i=1,\ldots,N}
							= \bigl( \zeta_i \bigr)_{i=1,\ldots,N} \in L^{\infty}(\Omega)^N.
						\end{equation*}
		\item[(iii)] 	There exists a Banach space $Z \subset X^N$ for which
						$Z \cap \mathcal{D} \bigl( \Omega; \, C_{\mathrm{per}}^{\infty} (Y)^N \bigr)$
						is dense in $Z$, and such that, for any
						$v \in Z \cap \mathcal{D} \bigl( \Omega; \, C_{\mathrm{per}}^{\infty} (Y)^N \bigr)$,
						\begin{equation*}
							\sum_{i=1}^N \sum_{j=1}^N \, \frac{\partial v_i}{\partial y_j}
								\bigl( \, \boldsymbol{\cdot} \, ,\alpha^h(\, \boldsymbol{\cdot} \,) \bigr) \,
									\frac{\partial \alpha^h_j}{\partial x_i} \rightharpoonup 0
										\quad \textrm{in } L^2 (\Omega).
						\end{equation*}
	\end{itemize}
\end{defin}
From Definition~\ref{def:typeH} and Proposition~\ref{prop:lambdatwoscalegrad}, we have the homogenization result
of Theorem~\ref{theo:lambdahomo}~\cite[Theorem~52]{SilPhD07}.
\begin{theo}
	\label{theo:lambdahomo}
	Assume that $\Omega$ has a Lipschitz continuous boundary and that
		\begin{equation*}
		A \in C_{\mathrm{per}} (Y)^{N \times N} \cap \mathcal{M} \bigl( r,s; \, \mathbb{R}^N \bigr).
	\end{equation*}
	Let $\{ \alpha^h\}$ be of type $\mathrm{H}^\zeta_{L^2(\Omega; \, C_{\mathrm{per}}(Y))}$.
	Then $\{ A \circ \alpha^h \}$ H-converges to $B$ given by
	\begin{equation*}
		\biggl( \sum_{j=1}^N \, b_{ij} \frac{\partial u}{\partial x_j} \biggr)_{i=1,\ldots,N}
			= \Biggl( \int\limits_Y \sum_{j=1}^N \, a_{ij}(y) \biggl( \frac{\partial u}{\partial x_j}
				+ w_{1,j}(\, \boldsymbol{\cdot} \,,y) \biggr) \, \mathrm{d} y \Biggr)_{i=1,\ldots,N},
	\end{equation*}
	$u \in H^1_0(\Omega)$ being the weak limit of the sequence $\{ u^h \}$ of solutions to \eqref{eq:PDEseq},
	if $u \in H^1_0(\Omega)$ and $w_1 \in Z^\perp \subset L^2 \bigl( \Omega; \, L^2_{\mathrm{per}}(Y)^N \bigr)$
	uniquely solve the homogenized problem
	\begin{equation}
		\label{eq:homoprob}
		\int\limits_{\Omega} \int\limits_Y \sum_{i=1}^N \sum_{j=1}^N \, a_{ij}(y)
			\biggl( \frac{\partial u}{\partial x_j}(x) + w_{1,j}(x,y) \biggr) \frac{\partial v}{\partial x_i}(x) \,
				\mathrm{d} y \, \mathrm{d} x \\
				= \int\limits_{\Omega} f(x) v(x) \, \mathrm{d} x
	\end{equation}
	for all $v \in H^1_0 (\Omega)$, and, for each $x \in \Omega$, the local problem
	\begin{equation}
		\label{eq:locprob}
		\sum_{i=1}^N \, \zeta_i(x) \int\limits_Y \sum_{j=1}^N \, a_{ij}(y) \biggl( \frac{\partial u}{\partial x_j}(x)
			+ w_{1,j}(x,y) \biggr) \frac{\partial v}{\partial y_i}(y) \, \mathrm{d} y = 0
	\end{equation}
	for all $v \in W_{\mathrm{per}} (Y)$.
\end{theo}
\begin{proof}
	First we make a weak formulation of the sequence of problems \eqref{eq:PDEseq}.

	The homogenized problem \eqref{eq:homoprob} is obtained by passing the weak formulation to the limit, using
	Proposition~\ref{prop:lambdatwoscalegrad}.

	The local problem \eqref{eq:locprob} is derived by first choosing appropriate test functions, then utilizing the
	fact that $\{ \alpha^h\}$ is of type $\mathrm{H}^\zeta_{L^2(\Omega; \, C_{\mathrm{per}}(Y))}$, and finally employing
	the variational lemma.
\end{proof}
\begin{rema}
	The original proof of Jeanette Silfver with all details is found in~\cite[Subsection~3.5.1]{SilPhD07}.
\end{rema}

\section{A non-periodic and two-dimensional example}
\label{sec:nonpertwodimex}

In order to justify the concept of homogenization of $\lambda$-structures, we take a look at a non-periodic
two-dimensional example (thus, fixing $N = 2$ from now on). Let $\bigl\{ \alpha^h \bigr\}$
and $\Omega$ be given by
\begin{equation}
	\label{eq:alphahexample}
	\left\{
		\begin{aligned}
				\alpha^h_1 (x) &= h x_1 \\
				\alpha^h_2 (x) &= h x_2 |x_2|
		\end{aligned}
	\right., \qquad x \in \mathbb{R}^2,
\end{equation}
and $\Omega = (a_1,b_1) \times (a_2,b_2)$, respectively, where $b_i > a_i > 0$, $i = 1,2$. Thus, $\Omega$ is an open
interval in the first quadrant of $\mathbb{R}^2$.
\begin{figure}[htp]
	\begin{center}
		\includegraphics[scale=0.7]{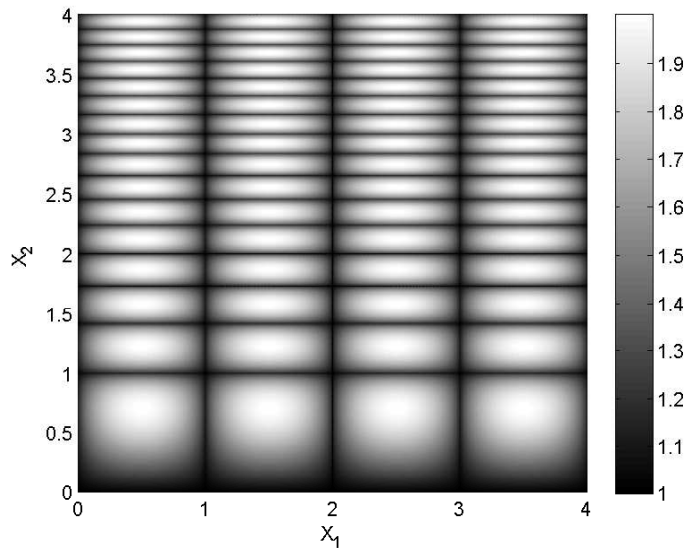}
	\end{center}
	\caption{Some entry of $A \circ \alpha^h$, $h = 1$, in the first quadrant.}
	\label{fig:twodimlambdastruct}
\end{figure}
In Figure~\ref{fig:twodimlambdastruct} we depict in the first quadrant the behaviour of some entry of $A \circ \alpha^h$
for some $(0,1)^2$-periodic matrix $A$. (To be specific, we have chosen an entry on the form
$1 + |\sin \pi y_1 \, \sin \pi y_2|^{1/2}$.) Note how the oscillation frequency increases with growing $x_2$, while the
periodicity is preserved in the $x_1$ direction, as expected.

To homogenize the sequence $\{ A \circ \alpha^h \}$, we must first check that $\{ \alpha^h \}$ is asymptotically
uniformly distributed on $\Omega$. Indeed, we have Proposition~\ref{prop:AUDexample}.
\begin{prop}
	\label{prop:AUDexample}
	The sequence $\{ \alpha^h \}$ given in \eqref{eq:alphahexample} is asymptotically uniformly distributed on
	$\Omega$.
\end{prop}
\begin{proof}
	We define $Y^{\hat{\jmath}} = \hat{\jmath} + Y$ and $Y_{\hat{k}} = \tfrac{1}{n} (\hat{k} + Y)$, where
	$\hat{\jmath} \in \mathbb{Z}^2$, $n \in \mathbb{Z}_+$ and $\hat{k} \in \{ 0,\ldots,n-1 \}^2$. Furthermore,
	$Y^{\hat{\jmath}}_{\hat{k}} = \hat{\jmath} + Y_{\hat{k}}$. Clearly, there exist bijections
	\begin{equation*}
		\left\{
			\begin{aligned}
				&\mathcal{J} \, : \, & \{ 1,2,\ldots \}		&\rightarrow \mathbb{Z}^2 \\
				&\mathcal{K} \, : \, & \{ 1,\ldots,n^2 \}	&\rightarrow \{ 0,\ldots,n-1 \}^2
			\end{aligned}
		\right.
	\end{equation*}
	such that $\{ \bar{Y}^{\mathcal{J}(j)} \}_{j=1}^{\infty}$ is a covering of $\mathbb{R}^2$, and
	$\bigl\{ \bar{Y}^{\mathcal{J}(j)}_{\mathcal{K}(k)} \bigr\}_{k=1}^{n^2}$ is a covering of
	$\bar{Y}^{\mathcal{J}(j)}$. This merely shows that the way of enumerating the squares in this proof is equivalent
	to the way in Definition~\ref{def:AUD}.

	We have, for each $\hat{\jmath} = (j_1,j_2) \in \mathbb{Z}^2$ and $\hat{k} = (k_1,k_2) \in \{ 0,\ldots,n-1 \}^2$,
	\begin{align*}
		\Omega^h_{\hat{\jmath}}			&= (\alpha^h)^{-1} (\bar{Y}^{\hat{\jmath}}) \cap \Omega =
			\Bigl( \bigl[ a_{1; \hat{\jmath}}^h, b_{1; \hat{\jmath}}^h \bigr] \times
				\bigl[  a_{2; \hat{\jmath}}^h, b_{2; \hat{\jmath}}^h \bigr] \Bigr) \cap \Omega \nonumber \\
		\intertext{and}
		\Omega^h_{\hat{\jmath},\hat{k}}	&= (\alpha^h)^{-1} (\bar{Y}^{\hat{\jmath}}_{\hat{k}}) \cap \Omega=
			\Bigl( \bigl[ a_{1; \hat{\jmath},\hat{k}}^h, b_{1; \hat{\jmath},\hat{k}}^h \bigr] \times
				\bigl[  a_{2; \hat{\jmath},\hat{k}}^h, b_{2; \hat{\jmath},\hat{k}}^h \bigr] \Bigr)
					\cap \Omega , \nonumber \\
	\end{align*}
	where
	\begin{align*}
		&\left\{
			\begin{aligned}
				a_{1; \hat{\jmath}}^h	&= \frac{j_1}{h},
					&\qquad	b_{1; \hat{\jmath}}^h	&= \frac{j_1+1}{h} \\
				a_{2; \hat{\jmath}}^h	&= \mathrm{sgn} (j_2) \sqrt{\frac{|j_2|}{h}},
					&\qquad	b_{2; \hat{\jmath}}^h	&= \mathrm{sgn} (j_2+1) \sqrt{\frac{|j_2+1|}{h}}
			\end{aligned}
		\right.
		\intertext{and}
		&\left\{
			\begin{aligned}
				a_{1; \hat{\jmath},\hat{k}}^h			&= \frac{j_1+\tfrac{k_1}{n}}{h},
					& b_{1; \hat{\jmath},\hat{k}}^h			&= \frac{j_1+\tfrac{k_1+1}{n}}{h}  \\
				a_{2; \hat{\jmath},\hat{k}}^h			&= \mathrm{sgn} \bigl( j_2 \! + \! \tfrac{k_2}{n} \bigr)
					\sqrt{\frac{\bigl| j_2 \! + \! \tfrac{k_2}{n} \bigr|}{h}}, 
					&	b_{2; \hat{\jmath},\hat{k}}^h		&= \mathrm{sgn} \bigl( j_2 \! + \! \tfrac{k_2+1}{n} \bigr)
					\sqrt{\frac{\bigl| j_2 \! + \! \tfrac{k_2+1}{n} \bigr|}{h}}
			\end{aligned}
		\right. \; .
	\end{align*}
	In order to cover $\Omega$, we can for each fixed $h$ apparently do this with a finite union of
	$\Omega^h_{\hat{\jmath}}$ sets. The smallest $j_1$ required for the covering will go like $\mathcal{O} (h)$ and the
	smallest $j_2$ like $\mathcal{O} \bigl( \sqrt{h} \, \bigr)$. For large $h$, the diameter of the
	$\Omega^h_{\hat{\jmath}}$ sets will go uniformly like $\mathcal{O} \bigl( \tfrac{1}{\sqrt{h}}\bigr)$, so we can fit
	them in balls of equal radius which tends to zero as $h \rightarrow \infty$. We realize that what is left to check
	is \eqref{eq:AUDequation}.

	For covering $\Omega^h_{\hat{\jmath}}$ sets, let $\hat{\jmath}$ be chosen such that
	$(\alpha^h)^{-1} (Y^{\hat{\jmath}})$ does not intersect $\partial \Omega$, i.e.,
	\begin{equation*}
		\bigl( a_{1; \hat{\jmath}}^h, b_{1; \hat{\jmath}}^h \bigr) \times
			\bigl(  a_{2; \hat{\jmath}}^h, b_{2; \hat{\jmath}}^h \bigr) \subset \Omega .
	\end{equation*}
	For such $\hat{\jmath}$'s we have the Lebesgue measures
	\begin{align*}
		\lambda \bigl( \Omega^h_{\hat{\jmath}} \bigr)			&= \frac{1}{h\sqrt{h}}
			\Bigl( \sqrt{j_2 + 1} - \sqrt{j_2} \, \Bigr)
		\intertext{and}
		\lambda \bigl( \Omega^h_{\hat{\jmath},\hat{k}} \bigr)	&= \frac{1}{h\sqrt{h}} \frac{1}{n}
			\Biggl( \sqrt{j_2 + \frac{k_2 + 1}{n}} - \sqrt{j_2 + \frac{k_2}{n}} \, \Biggr) .
	\end{align*}
	Thus,
	\begin{align*}
		\frac{\lambda \bigl( \Omega^h_{\hat{\jmath},\hat{k}} \bigr)}{\lambda \bigl( \Omega^h_{\hat{\jmath}} \bigr)}
			&= \frac{1}{n} \frac{\sqrt{j_2 + \dfrac{k_2 + 1}{n}}
				- \sqrt{j_2 + \dfrac{k_2}{n}}}{\sqrt{j_2 + 1} - \sqrt{j_2}} \\
			&= \frac{1}{n^2} \frac{\sqrt{j_2 + 1} + \sqrt{j_2}}{\sqrt{j_2 +
				\dfrac{k_2 + 1}{n}} + \sqrt{j_2 + \dfrac{k_2}{n}}} .
	\end{align*}
	Since the smallest $j_2$ goes like $\mathcal{O} \bigl( \sqrt{h} \, \bigr)$, we must uniformly have that
	\begin{equation*}
		\frac{\lambda \bigl( \Omega^h_{\hat{\jmath},\hat{k}} \bigr)}{\lambda \bigl( \Omega^h_{\hat{\jmath}} \bigr)}
			\rightarrow \frac{1}{n^2} = \lambda ( Y_{\hat{k}} ) ,
	\end{equation*}
	for any $\hat{k} \in \{ 0,\ldots,n-1 \}^2$, as $h \rightarrow \infty$. This implies
	\begin{equation*}
		\left| \frac{\lambda(\Omega^h_{j,k})}{\lambda(\Omega^h_j)} - \lambda(Y_k) \right| < \epsilon ,
	\end{equation*}
	where $\epsilon = \epsilon (h) \rightarrow 0$ as $h \rightarrow \infty$.
\end{proof}
In virtue of Proposition~\ref{prop:lambdastronglytwoscale}, we thus get
Proposition~\ref{prop:lambdastronglytwoscaleex}.
\begin{prop}
	\label{prop:lambdastronglytwoscaleex}
	The sequence $\{ \tau^h \}$ defined by
	\begin{equation*}
		\tau^h v = v \bigl( \, \boldsymbol{\cdot} \,, \alpha^h (\, \boldsymbol{\cdot} \,) \bigr) ,
	\end{equation*}
	where $\{ \alpha^h \}$ is given by \eqref{eq:alphahexample}, is strongly two-scale compatible with respect to
	$L^2 \bigl( \Omega; \, C_{\mathrm{per}} (Y) \bigr)$.
\end{prop}
\begin{proof}
	Use Proposition~\ref{prop:AUDexample} together with Proposition~\ref{prop:lambdastronglytwoscale}.
\end{proof}
Define the $2$-tuple $\zeta$ according to
\begin{equation}
	\label{eq:diagonalmatrix}
	\left\{
		\begin{aligned}
			\zeta_1(x) &= 1 \\
			\zeta_2(x) &= 2x_2
		\end{aligned}
	\right. ,
\end{equation}
and define the diagonal matrix $\Pi$ element-wise by
\begin{equation*}
	\bigl( \pi_{ii}(x)\bigr)_{i=1,2} = \bigl( \zeta_i(x) \bigr)_{i=1,2} .
\end{equation*}
Proposition~\ref{prop:typeHex} below shows that $\{ \alpha^h \}$ is of right type to use in
Theorem~\ref{theo:lambdahomo} for the homogenization to work out properly.
\begin{prop}
	\label{prop:typeHex}
	The sequence $\{ \alpha^h \}$ given by \eqref{eq:alphahexample} is of type
	$\mathrm{H}_{L^2( \Omega; \, C_{\mathrm{per}} (Y))}^\zeta$, where $\zeta$ is given by \eqref{eq:diagonalmatrix}.
\end{prop}
\begin{proof}
	We must check conditions (i)--(iii) of Definition~\ref{def:typeH}. Condition (i) is an immediate consequence of
	Proposition~\ref{prop:lambdastronglytwoscaleex}. For condition (ii), we note that
	\begin{equation*}
		\left\{
			\begin{aligned}
				\frac{\partial \alpha^h_1}{\partial x_1}(x) &= h \\
				\frac{\partial \alpha^h_2}{\partial x_2}(x) &= 2 h x_2 \\
				\frac{\partial \alpha^h_j}{\partial x_i}(x) &= 0, \quad i \neq j
			\end{aligned}
		\right. ,
	\end{equation*}
	and by choosing
	\begin{equation*}
		p^h (x) = \frac{1}{h} ,
	\end{equation*}
	for which $p^h \rightarrow 0$ in $H^1(\Omega)$, we get
	\begin{equation*}
		p^h \, \Bigl( \frac{\partial \alpha^h_j}{\partial x_i} \Bigr)_{i,j=1,2} = \bigl( \pi_{ij} \big)_{i,j=1,2}
			\rightarrow \bigl( \pi_{ij} \big)_{i,j=1,2} \quad \textrm{in } L^2(\Omega).
	\end{equation*}
	We obviously have that $\zeta \in L^{\infty} (\Omega)^2$. Finally, we must check condition (iii) of
	Definition~\ref{def:typeH}. In this context, it means that we must find a Banach space
	$Z \subset L^2 \bigl( \Omega; \, C_{\mathrm{per}} (Y)^2 \bigr)$ for which
	$Z \cap \mathcal{D} \bigl( \Omega; \, C_{\mathrm{per}}^{\infty} (Y)^2 \bigr)$ is dense in $Z$, and such that, for any
	$v \in Z \cap \mathcal{D} \bigl( \Omega; \, C_{\mathrm{per}}^{\infty} (Y)^2 \bigr)$,
	\begin{equation}
		\label{eq:weakconvzeta}
		h \sum_{i=1}^2 \, \frac{\partial v_i}{\partial y_i}\bigl( \, \boldsymbol{\cdot} \, ,
			\alpha^h(\, \boldsymbol{\cdot} \,) \bigr) \, \zeta_i \rightharpoonup 0 \quad \textrm{in } L^2 (\Omega),
	\end{equation}
	as $h \rightarrow \infty$. If we define
	\begin{equation}
		\label{eq:Zspace}
		Z = \Bigl\{ v \in L^2 \bigl( \Omega; \, L_{\mathrm{per}}^2 (Y)^2 \bigr) \, : \, \sum_{i=1}^2 \,
			\frac{\partial v_i}{\partial y_i} \, \zeta_i = 0 \Bigr\},
	\end{equation}
	we clearly have, for any $v \in Z \cap \mathcal{D} \bigl( \Omega; \, C_{\mathrm{per}}^{\infty} (Y)^2 \bigr)$, a
	satisfied weak convergence \eqref{eq:weakconvzeta}. What is left to prove is the density of
	$Z \cap \mathcal{D} \bigl( \Omega; \, C_{\mathrm{per}}^{\infty} (Y)^2 \bigr)$ in $Z$, where $Z$ is given by
	\eqref{eq:Zspace}. We apparently have
	\begin{equation}
		\label{eq:ZDspace}
		Z \cap \mathcal{D} \bigl( \Omega; \, C_{\mathrm{per}}^{\infty} (Y)^2 \bigr)
			= \Bigl\{ v \in \mathcal{D} \bigl( \Omega; \, C_{\mathrm{per}}^{\infty} (Y)^2 \bigr) \,
				: \, \sum_{i=1}^2 \, \frac{\partial v_i}{\partial y_i} \, \zeta_i = 0 \Bigr\},
	\end{equation}
	and since $\mathcal{D} \bigl( \Omega; \, C_{\mathrm{per}}^{\infty} (Y)^2 \bigr)$ is dense in
	$L^2 \bigl( \Omega; \, L_{\mathrm{per}}^2 (Y)^2 \bigr)$, it is easy to check that we have the desired density of
	$Z \cap \mathcal{D} \bigl( \Omega; \, C_{\mathrm{per}}^{\infty} (Y)^2 \bigr)$ in $Z$.
\end{proof}
\begin{rema}
	We have silently used the fact that the vanishing divergence in the definition for $Z$ and the expression for
	$Z \cap \mathcal{D} \bigl( \Omega; \, C_{\mathrm{per}}^{\infty} (Y)^2 \bigr)$ in \eqref{eq:Zspace} and
	\eqref{eq:ZDspace}, respectively, does not upset the inheritance of the density property for
	$\mathcal{D} \bigl( \Omega; \, C_{\mathrm{per}}^{\infty} (Y)^2 \bigr)$ in
	$L^2 \bigl( \Omega; \, L_{\mathrm{per}}^2 (Y)^2 \bigr)$.
\end{rema}
To obtain the orthogonal complement, we need Lemma~\ref{lem:orthcompl}~\cite{Ngu89}.
\begin{lemma}
	\label{lem:orthcompl}
	Assume that $Y$ is a unit cube in $\mathbb{R}^N$. Let $f \in L^2_{\mathrm{per}} (Y)^N$ be orthogonal to the space
	\begin{equation*}
		\Bigl\{ g \in C_{\mathrm{per}}^{\infty} (Y)^N \, :
			\, \sum_{i=1}^N \, \frac{\partial g_i}{\partial y_i} = 0 \Bigr\}
	\end{equation*}
	of divergence free functions. Then, for some $h \in W_{\mathrm{per}} (Y)$,
	\begin{equation*}
		\bigl( f_i \bigr)_{i=1,\ldots,N} = \Bigl( \frac{\partial h}{\partial y_i} \Bigr)_{i=1,\ldots,N}.
	\end{equation*}
\end{lemma}
We are now ready to characterize the orthogonal complement.
\begin{prop}
	\label{prop:orthocomplex}
	The orthogonal complement $Z^\perp \subset L^2 \bigl( \Omega; \, L_{\mathrm{per}}^2 (Y)^2 \bigr)$ of $Z$, defined in
	the proof of Proposition~\ref{prop:typeHex}, is
	\begin{equation*}
		Z^\perp = \Bigl\{ \Bigl( \zeta_i \frac{\partial u_1}{\partial y_i} \Bigr)_{i=1,2} \, :
			\, u_1 \in L^2 \bigl( \Omega; \, W_{\mathrm{per}} (Y) \bigr) \Bigr\} .
	\end{equation*}
\end{prop}
\begin{proof}
	Suppose $v \in Z$ and $w_1 \in Z^\perp$. Then, by definition,
	\begin{align*}
		0	&= \int\limits_{\Omega} \int\limits_Y \sum_{i=1}^2 \,
				v_i(x,y) w_{1,i}(x,y) \, \mathrm{d} y \, \mathrm{d} x \\
			&= \int\limits_{\Omega} \int\limits_Y \sum_{i=1}^2 \,
				v_i(x,y) \zeta_i(x) \frac{w_{1,i}(x,y)}{\zeta_i(x)} \, \mathrm{d} y \, \mathrm{d} x.
	\end{align*}
	Since
	\begin{equation*}
		\Bigl( \frac{w_{1,i} (x,\, \boldsymbol{\cdot} \,)}{\zeta_i(x)} \Bigr)_{i=1,2}
			\in L_{\mathrm{per}}^2 (Y)^2 \quad \textrm{a.e. } x \in \Omega ,
	\end{equation*}
	Lemma~\ref{lem:orthcompl} implies that, for some
	$u_1 (x,\, \boldsymbol{\cdot} \,) \in W_{\mathrm{per}} (Y)$ a.e.~$x \in \Omega$,
	\begin{equation*}
		\Bigl( \frac{w_{1,i}(x,\, \boldsymbol{\cdot} \,)}{\zeta_i(x)} \Bigr)_{i=1,2}
		= \Bigl( \frac{\partial u_1}{\partial y_i}(x,\, \boldsymbol{\cdot} \,) \Bigr)_{i=1,2}
			\quad \textrm{a.e. } x \in \Omega.
	\end{equation*}
	Hence, for some $u_1 (x,\, \boldsymbol{\cdot} \,) \in W_{\mathrm{per}} (Y)$ a.e.~$x \in \Omega$,
	\begin{equation*}
		\bigl( w_{1,i}(x,\, \boldsymbol{\cdot} \,) \bigr)_{i=1,2}
		= \Bigl( \zeta_i(x) \frac{\partial u_1}{\partial y_i}(x,\, \boldsymbol{\cdot} \,) \Bigr)_{i=1,2}
			\quad \textrm{a.e. } x \in \Omega.
	\end{equation*}
	We know that $Z^\perp \subset L^2 \bigl( \Omega; \, L_{\mathrm{per}}^2 (Y)^2 \bigr)$, so
	$w_1\in L^2 \bigl( \Omega; \, L_{\mathrm{per}}^2 (Y)^2 \bigr)$ holds. This implies
	\begin{align*}
		\lVert u_1 \rVert_{L^2 ( \Omega; \, W_{\mathrm{per}} (Y) )}
					&= \lVert \nabla_y u_1 \rVert_{L^2 ( \Omega; \, L_{\mathrm{per}}^2 (Y)^2 )} \\
					&\leqslant \max \bigl\{ 1, \tfrac{1}{2a_2} \bigr\}
						\lVert w_1 \rVert_{L^2 ( \Omega; \, L_{\mathrm{per}}^2 (Y)^2 )} \\
					&< \infty,
	\end{align*}
	where we have recalled that $\Omega = (a_1,b_1) \times (a_2,b_2)$. Thus,
	$u_1 \in L^2 \bigl( \Omega; \, W_{\mathrm{per}} (Y) \bigr)$, and we conclude that
	\begin{equation*}
		Z^\perp = \Bigl\{ \Bigl( \zeta_i \frac{\partial u_1}{\partial y_i} \Bigr)_{i=1,2} \, :
			\, u_1 \in L^2 \bigl( \Omega; \, W_{\mathrm{per}} (Y) \bigr) \Bigr\},
	\end{equation*}
	and we are done.
\end{proof}
We can now formulate a preliminary homogenization result in Proposition~\ref{prop:lambdahomoex}.
\begin{prop}
	\label{prop:lambdahomoex}
	Assume
	\begin{equation*}
		A \in C_{\mathrm{per}} (Y)^{2 \times 2} \cap \mathcal{M} \bigl( r,s; \, \mathbb{R}^2 \bigr)
	\end{equation*}
	and let $\{ \alpha^h \}$ and $\zeta$ be given by \eqref{eq:alphahexample} and \eqref{eq:diagonalmatrix},
	respectively.
	Then $\{ A \circ \alpha^h\}$ H-converges to $B$ given by
	\begin{equation}
		\label{eq:flowhomoex}
		\Bigl( \sum_{j=1}^2 \, b_{ij} \frac{\partial u}{\partial x_j} \Bigr)_{i=1,2}
		= \Biggl( \int\limits_Y \sum_{j=1}^2 \, a_{ij}(y) \Bigl( \frac{\partial u}{\partial x_j}
			+ \zeta_j \frac{\partial u_1}{\partial y_j}(\, \boldsymbol{\cdot} \, ,y) \Bigr)
				\, \mathrm{d} y \Biggr)_{i=1,2} ,
	\end{equation}
	$u \in H^1_0 (\Omega)$ being the weak limit of the sequence $\{ u^h \}$ of solutions to \eqref{eq:PDEseq},
	if $u \in H^1_0(\Omega)$ and $u_1 \in L^2 \bigl( \Omega; \, W_{\mathrm{per}} (Y) \bigr)$
	uniquely solve the homogenized problem
	\begin{multline*}
		\int\limits_{\Omega} \int\limits_Y \sum_{i=1}^2 \sum_{j=1}^2 \, a_{ij}(y)
			\biggl( \frac{\partial u}{\partial x_j}(x) + \zeta_j(x) \frac{\partial u_1}{\partial y_j}(x,y) \biggr)
				\frac{\partial v}{\partial x_i}(x) \, \mathrm{d} y \, \mathrm{d} x \\
					= \int\limits_{\Omega} f(x) v(x) \, \mathrm{d} x
	\end{multline*}
	for all $v \in H^1_0 (\Omega)$, and, for each $x \in \Omega$, the local problem
	\begin{equation}
		\label{eq:localprob}
		\int\limits_Y \sum_{i=1}^2 \sum_{j=1}^2 \, a_{ij}(y) \biggl( \frac{\partial u}{\partial x_j}(x)
			+ \zeta_j(x) \frac{\partial u_1}{\partial y_j}(x,y) \biggr)
				\zeta_i(x) \frac{\partial v}{\partial y_i}(y) \, \mathrm{d} y = 0
	\end{equation}
	for all $v \in W_{\mathrm{per}} (Y)$.
\end{prop}
\begin{proof}
	This is an immediate consequence of Theorem~\ref{theo:lambdahomo} together with Proposition~\ref{prop:typeHex} and
	Proposition~\ref{prop:orthocomplex}.
\end{proof}
It is possible to improve Proposition~\ref{prop:lambdahomoex} to yield an explicit homogenized matrix and a local
problem of the same type as in Theorem~\ref{theo:elliptichomo}. Indeed, we have Theorem~\ref{theo:lambdahomoex}.
\begin{theo}
	\label{theo:lambdahomoex}
	Assume
	\begin{equation*}
		A \in C_{\mathrm{per}} (Y)^{2 \times 2} \cap \mathcal{M} \bigl( r,s; \, \mathbb{R}^2 \bigr)
	\end{equation*}
	and let $\{ \alpha^h \}$ and $\zeta$ be given by \eqref{eq:alphahexample} and \eqref{eq:diagonalmatrix},
	respectively.
	Then $\{ A \circ \alpha^h\}$ H-converges to $B$ given by
	\begin{equation}
		\label{eq:homogenizedmatrix}
		\bigl( b_{ij} \bigr)_{i,j=1,2} = \Biggl( \int\limits_Y \sum_{k=1}^2 \, a_{ik}(y) \Bigl( \delta_{kj}
			+ \zeta_k \frac{\partial z_j}{\partial y_k}(\, \boldsymbol{\cdot} \, ,y) \Bigr)
				\, \mathrm{d} y \Biggr)_{i=1,2},
	\end{equation}
	where $z \in L^{\infty} \bigl( \Omega; \, W_{\mathrm{per}} (Y)^2 \bigr)$ uniquely solves, for each
	$x \in \Omega$, the local problem
	\begin{equation}
		\label{eq:localproblemex}
		- \Biggl( \sum_{i=1}^2\sum_{k=1}^2 \, \zeta_i(x) \frac{\partial}{\partial y_i} \! \biggl\lgroup \! a_{ik}
			\Bigl( \delta_{kj} + \zeta_k(x) \frac{\partial z_j}{\partial y_k}(x,\, \boldsymbol{\cdot} \,) \Bigr)
				\! \biggr\rgroup \! \Biggr)_{j=1,2} = 0 \quad \textrm{in } Y.
	\end{equation}
\end{theo}
$\phantom{This is just to skip one line.}$
\begin{proof}
	The proof will be performed in four steps, where the first step introduces an ansatz, the second and third steps
	derive the homogenized matrix \eqref{eq:homogenizedmatrix} and the local problem \eqref{eq:localproblemex},
	respectively. In the last step we prove the uniqueness of the solution to the local problem. \\

	\emph{Step (i): Ansatz.} Let $u \in H^1_0(\Omega)$ and
	$u_1 \in L^2 \bigl( \Omega; \, W_{\mathrm{per}} (Y) \bigr)$ solve the system of equations in
	Proposition~\ref{prop:lambdahomoex}, and make the ansatz
	\begin{equation}
		\label{eq:ansatz}
		u_1 = \sum_{j=1}^2 \, \frac{\partial u}{\partial x_j} z_j,
	\end{equation}
	where $z \in L^{\infty} \bigl( \Omega; \, W_{\mathrm{per}} (Y)^2 \bigr)$. We will see later why it is
	necessary that we must constrain ourselves to $z(\, \boldsymbol{\cdot} \,,y) \in L^{\infty} (\Omega)^2$
	a.e.~$y \in \mathbb{R}^2$. \\

	\emph{Step (ii): Homogenized matrix.} Let us first derive the expression \eqref{eq:homogenizedmatrix} for the
	homogenized matrix. From \eqref{eq:flowhomoex} and \eqref{eq:ansatz} we get
	\begin{equation*}
		\biggl( \sum_{j=1}^2 \, b_{ij} \frac{\partial u}{\partial x_j} \biggr)_{i=1,2}
		= \Biggl( \sum_{j=1}^2 \, \int\limits_Y \sum_{k=1}^2 \, a_{ik}(y) \Bigl( \delta_{kj}
			+ \zeta_k \frac{\partial z_j}{\partial y_k}(\, \boldsymbol{\cdot} \, ,y) \Bigr) \, \mathrm{d} y
				\, \frac{\partial u}{\partial x_j} \Biggr)_{i=1,2},
	\end{equation*}
	which is satisfied if
	\begin{equation*}
		\bigl( b_{ij} \bigr)_{i,j=1,2} = \Biggl( \int\limits_Y \sum_{k=1}^2 \, a_{ik}(y) \Bigl( \delta_{kj}
			+ \zeta_k \frac{\partial z_j}{\partial y_k}(\, \boldsymbol{\cdot} \, ,y) \Bigr)
				\, \mathrm{d} y \Biggr)_{i,j=1,2},
	\end{equation*}
	where we note that $B \in L^{\infty} (\Omega)^{2 \times 2}$, which requires
	$z(\, \boldsymbol{\cdot} \, ,y) \in L^{\infty} (\Omega)^2$ a.e.~$y \in \mathbb{R}^2$. \\

	\emph{Step (iii): Local problem.} Next, let us derive the local problem \eqref{eq:localproblemex}. Fix some
	$v \in W_{\mathrm{per}} (Y)$ to be used in the local problem \eqref{eq:localprob} in
	Proposition~\ref{prop:lambdahomoex}, whose left-hand side becomes, for each $x \in \Omega$,
	\begin{multline*}
		\int\limits_Y \sum_{i=1}^2 \sum_{k=1}^2 \, a_{ik}(y) \! \biggl\lgroup \! \frac{\partial u}{\partial x_k}(x)
			+ \zeta_k(x) \frac{\partial}{\partial y_k} \sum_{j=1}^2 \, \frac{\partial u}{\partial x_j}(x) z_j(x,y)
				\! \biggr\rgroup \! \zeta_i(x) \frac{\partial v}{\partial y_i}(y) \, \mathrm{d} y \\
		= \sum_{j=1}^2 \, \int\limits_Y \sum_{i=1}^2 \sum_{k=1}^2 \, \zeta_i(x) \frac{\partial v}{\partial y_i}(y)
			\, a_{ik}(y) \Bigl( \delta_{kj} + \zeta_k(x) \frac{\partial z_j}{\partial y_k}(x,y) \Bigr) \, \mathrm{d} y
				\, \frac{\partial u}{\partial x_j}(x) .
	\end{multline*}
	This must be zero, which is the case if, for each $x \in \Omega$,
	\begin{equation*}
		\Biggl( \int\limits_Y \sum_{i=1}^2 \sum_{k=1}^2 \, \zeta_i(x) \frac{\partial v}{\partial y_i}(y) \, a_{ik}(y)
			\Bigl( \delta_{kj} + \zeta_k(x) \frac{\partial z_j}{\partial y_k}(x,y) \Bigr)
				\, \mathrm{d} y \Biggr)_{j=1,2} = 0 .
	\end{equation*}
	By partial integrating and using the divergence theorem, we obtain, for each $x \in \Omega$,
	\begin{multline*}
		\Biggl( \; \int\limits_{\partial Y} \sum_{i=1}^2 \sum_{k=1}^2 \, n_i(y) \, \zeta_i(x) v(y) a_{ik}(y)
			\Bigl( \delta_{kj} + \zeta_k(x) \frac{\partial z_j}{\partial y_k}(x,y) \Bigr)
				\, \mathrm{d} S \Biggr)_{j=1,2} \\
		- \Biggl( \int\limits_Y v(y) \sum_{i=1}^2 \sum_{k=1}^2 \, \zeta_i(x) \frac{\partial}{\partial y_i}
			\! \biggl\lgroup \! a_{ik}(y) \Bigl( \delta_{kj} + \zeta_k(x) \frac{\partial z_j}{\partial y_k}(x,y) \Bigr)
				\! \biggr\rgroup \! \, \mathrm{d} y \Biggr)_{j=1,2} = 0,
	\end{multline*} 
	where $n$ is the unit outward normal to $\partial Y$. Since $v$ is $Y$-periodic, the surface integral vanishes, and
	we are left with, for each $x \in \Omega$,
	\begin{equation*}
		- \Biggl( \int\limits_Y v(y) \sum_{i=1}^2 \sum_{k=1}^2 \, \zeta_i(x) \frac{\partial}{\partial y_i}
			\! \biggl\lgroup \! a_{ik}(y) \Bigl( \delta_{kj} + \zeta_k(x) \frac{\partial z_j}{\partial y_k}(x,y) \Bigr)
				\! \biggr\rgroup \! \, \mathrm{d} y \Biggr)_{j=1,2} = 0,
	\end{equation*}
	which certainly is satisfied if, for each $x \in \Omega$,
	\begin{equation*}
		- \Biggl( \sum_{i=1}^2 \sum_{k=1}^2 \, \zeta_i(x) \frac{\partial}{\partial y_i} \! \biggl\lgroup \! a_{ik}
			\Bigl( \delta_{kj} + \zeta_k(x) \frac{\partial z_j}{\partial y_k} \Bigr)
				\! \biggr\rgroup \! \Biggr)_{j=1,2} = 0 ,
	\end{equation*}
	and we have shown \eqref{eq:localproblemex}. \\

	\emph{Step (iv): Uniqueness.} It remains to prove the uniqueness of the solution to the local problem. Fixing
	$x \in \Omega$, we can define a new, rescaled, $y$-variable $y^{\zeta(x)}$ by letting
	\begin{equation*}
		\bigl( y^{\zeta(x)}_i (y) \bigr)_{i=1,2} = \Bigl(  \frac{y_i}{\zeta_i(x)} \Bigr)_{i=1,2}, \qquad y \in Y .
	\end{equation*}
	The new local problem can be written
	\begin{equation*}
		- \Biggl( \sum_{i=1}^2\sum_{k=1}^2 \, \frac{\partial}{\partial y^{\zeta(x)}_i} \! \biggl\lgroup \!
			a^{\zeta(x)}_{ik} \biggl( \delta_{kj} + \frac{\partial z^{\zeta(x)}_j}{\partial y^{\zeta(x)}_k} \biggr)
				\! \biggr\rgroup \! \Biggr)_{j=1,2} = 0 \quad \textrm{in } Y^{\zeta(x)},
	\end{equation*}
	effectively a ``classical'' local problem, where the $2 \times 2$ matrix $A^{\zeta(x)}$ and the $2$-tuple
	$z^{\zeta(x)}$ are given according to
	\begin{equation*}
		\left\{
			\begin{aligned}
				A^{\zeta(x)} \circ y^{\zeta(x)}	&= A \\
				z^{\zeta(x)} \circ y^{\zeta(x)} &= z(x,\, \boldsymbol{\cdot} \,)
			\end{aligned}
		\right. \; ,
	\end{equation*}
	and $Y^{\zeta(x)} = (0,1) \times \bigl( 0,\tfrac{1}{2x_2} \bigr)$. Since
	\begin{equation*}
		A^{\zeta(x)} \in C_{\mathrm{per}} \bigl( Y^{\zeta(x)} \bigr)^{2 \times 2}
			\cap \mathcal{M} \bigl( r,s; \, \mathbb{R}^2 \bigr) ,
	\end{equation*}
	uniqueness is ensured just as in the ``classical'' case.
\end{proof}
\begin{rema}
	In Step (iv), the fact that $Y^{\zeta(x)}$, for each $x \in \Omega$, is a rectangle rather than the unit cube will,
	of course, not spoil our  argumentation.
\end{rema}
\begin{figure}[htp]
	\begin{center}
		\includegraphics[scale=0.7]{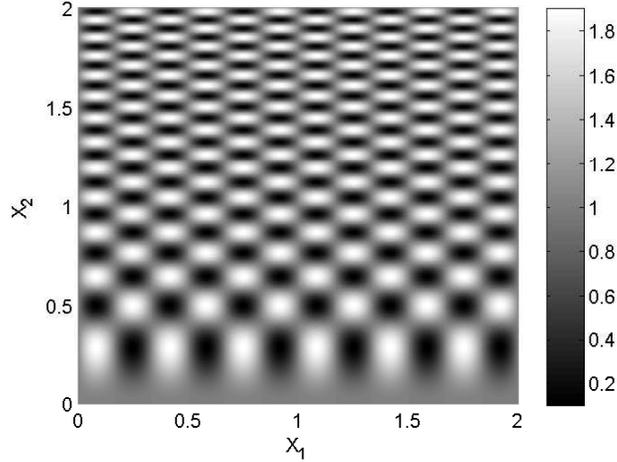}
	\end{center}
	\caption{The scalar factor of $A \circ \alpha^h$, $h = 3$.}
	\label{fig:scalarfactor}
\end{figure}
\begin{figure}[htp]
	\begin{center}
		\includegraphics[scale=0.7]{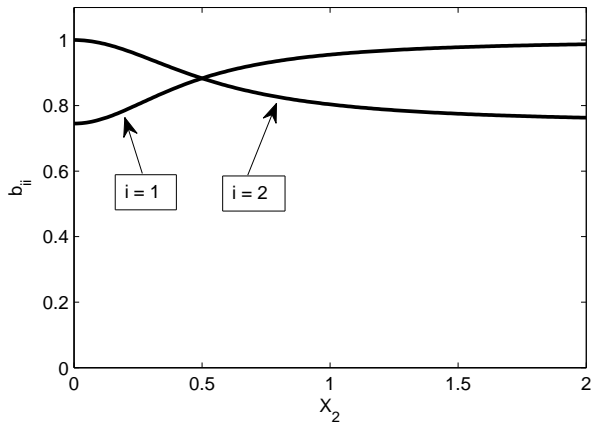}
	\end{center}
	\caption{The diagonal entries of $B$.}
	\label{fig:homomatrixcomp}
\end{figure}

\section{A numerical illustration}
\label{sec:numill}
	
As an illustration of the theoretical results of Section~\ref{sec:nonpertwodimex}, consider the sequence of problems
\eqref{eq:PDEseq} with $A^h = A \circ \alpha^h$ where $A$ is given as a product between a scalar factor and a unit matrix according to
\begin{equation*}
	\bigl( a_{ij}(y) \bigr)_{i,j=1,2} = \bigl( 1 + \tfrac{9}{10} \sin 2 \pi y_1 \, \sin 2 \pi y_2 \bigr)
		( \delta_{ij} )_{i,j=1,2}, \quad y \in \mathbb{R}^2,
\end{equation*}
and let $\Omega = (\delta,2)^2$ where $\delta \gtrsim 0$. Apparently,
\begin{equation*}
	A \in C_{\mathrm{per}} (Y)^{2 \times 2} \cap \mathcal{M} \bigl( \tfrac{19}{10},\tfrac{1}{10}; \, \mathbb{R}^2 \bigr),
\end{equation*}
so Theorem~\ref{theo:lambdahomoex} is applicable. Furthermore,
\begin{equation*}
	\bigl( ( a_{ij} \circ \alpha^h ) (x) \bigr)_{i,j=1,2} = \bigl( 1 + \tfrac{9}{10} \sin 2 \pi hx_1
		\, \sin 2 \pi hx_2^2 \bigr) ( \delta_{ij} )_{i,j=1,2}, \quad x \in (\delta,2)^2,
\end{equation*}
see Figure~\ref{fig:scalarfactor}. By using \eqref{eq:localproblemex}, the local problem for $z$ becomes, for each
$x \in \Omega$,
\begin{equation}
	\label{eq:heateqlocprob}
	- \biggl( \sum_{i=1}^2 \sum_{j=1}^2 \, \frac{\partial}{\partial y_i} \Bigl( c_{ij}(x,\, \boldsymbol{\cdot} \,)
		\frac{\partial z_k}{\partial y_j}(x,\, \boldsymbol{\cdot} \,) \Bigr) \biggr)_{k=1,2}
	= \bigl( g_k(x,\, \boldsymbol{\cdot} \,) \bigr)_{k=1,2} \quad \textrm{in } Y,
\end{equation}
where the matrix $C$ and the $2$-tuple $g$ are given by
\begin{align*}
	&\left\{
		\begin{aligned}
			c_{11}(x,y)	&= \bigl( 1 + \tfrac{9}{10} \sin 2 \pi y_1 \, \sin 2 \pi y_2 \bigr) \\
			c_{22}(x,y)	&= 4x_2^2 \, \bigl( 1 + \tfrac{9}{10} \sin 2 \pi y_1 \, \sin 2 \pi y_2 \bigr) \\
			c_{ij}(x,y)	&= 0, \quad i \neq j
		\end{aligned}
	\right.
	\intertext{and}
	&\left\{
		\begin{aligned}
			g_1(x,y)	&= \tfrac{9}{5} \pi \, \cos 2 \pi y_1 \, \sin 2 \pi y_2 \\
			g_2(x,y)	&= \tfrac{18}{5} \pi x_2 \, \sin 2 \pi y_1 \, \cos 2 \pi y_2
		\end{aligned}
	\right. ,
\end{align*}
respectively. Solving \eqref{eq:heateqlocprob} numerically (effectively we have a one-parameter family, with respect to
$x_2$, of partial differential equations to solve) and then computing the homogenized matrix through
\eqref{eq:homogenizedmatrix}, we get that $B$ is diagonal with non-vanishing entries, functions with respect to $x_2$
only, given according to Figure~\ref{fig:homomatrixcomp}. Note the interesting feature that
$b_{11}|_{x_2 = 1/2} = b_{22}|_{x_2 = 1/2}$, where $B$ obviously is proportional to the unit matrix, i.e., along the
line $x_2 = \tfrac{1}{2}$, the homogenized matrix is isotropic. The heuristic explanation to this is simple; for large
$h$ the mapped periodicity cells in the vicinity of the line $x_2 = \tfrac{1}{2}$ are near-perfect squares. Of course,
it is crucial that $A$ is isotropic to begin with in order for the map $A \circ \alpha^h$ to exhibit a near-isotropy
property on such mapped, near-perfect squares.

\end{document}